\documentclass[10pt]{amsart}

\usepackage{amsmath,amssymb,amsthm,amsfonts,pgf}
\usepackage[square,numbers]{natbib}
\usepackage{url}

\usepackage{color}

\newtheorem{thm}{Theorem}[section]

\newtheorem{lem}[thm]{Lemma}
\newtheorem{prop}[thm]{Proposition}

\newtheorem{clm}[thm]{Claim}

\theoremstyle{remark}

\newtheorem{rmk}[thm]{Remark}

\newtheorem{ntn}[thm]{Notation}

\newtheorem{obs}[thm]{Observation}

\theoremstyle{definition}
\newtheorem{dfn}[thm]{Definition}
\newtheorem{Def}[thm]{Definition}

\newcommand{\la}{\left <}
\newcommand{\ra}{\right >}
\newcommand{\uphp}{\upharpoonright}
\newcommand{\ov}{\overline}

\newcommand{\pl}{\shortparallel}
\newcommand{\cdd}{\cdots}

\newcommand{\etb}{\bar{\eta}}
\newcommand{\nub}{\bar{\nu}}

\newcommand{\leftexp}[2]{{\vphantom{#2}}^{#1}{#2}}

\newcommand{\W}{\leftexp{\omega >}{\omega}}
\newcommand{\B}{\leftexp{\beta >}{\lambda}}

\newcommand{\CM}{\mathcal{M}}

\newcommand{\nin}{\noindent}

\newcommand{\be}{\begin{enumerate}}
\newcommand{\ee}{\end{enumerate}}

\newcommand{\ssk}{\smallskip}
\newcommand{\bs}{\bigskip}
\newcommand{\ms}{\medskip}

\newcommand{\ya}{\vDash}
\newcommand{\imp}{\rightarrow}

\newcommand{\sbb}{\subseteq}

\newcommand{\lteq}{\trianglelefteq}

\newcommand{\al}{\alpha}
\newcommand{\kk}{\kappa}

\newcommand{\lam}{\lambda}
\newcommand{\gam}{\gamma}
\newcommand{\og}{\omega}

\newcommand{\xb}{\bar{x}}
\newcommand{\yb}{\bar{y}}

\newcommand{\cb}{\bar{c}}

\newcommand{\fr}{\smallfrown}
\newcommand{\s}{\textrm{s}}
\newcommand{\n}{\textrm{str}}

\newcommand{\len}{<_{\textrm{len}}}
\newcommand{\lx}{<_{\textrm{lex}}}
\newcommand{\qt}{\textrm{qftp}}
\newcommand{\ar}{\rightarrow}
\newcommand{\I}{\mathcal{I}}
\newcommand{\bij}{\hookrightarrow}
\newcommand{\gm}{\gamma}

\def\tp{\operatorname{tp}}

\newcommand{\cc}{\mathbf{c}}

\title{Tree indiscernibilities, revisited}
\author{Byunghan Kim}
\thanks{The first author was supported by an NRF grant 2011-0021916.}

\author{Hyeung-Joon Kim}
\thanks{The second author was supported by the second phase of the Brain Korea 21 Program in 2011.}

\author{Lynn Scow}
\thanks{The third author was supported by the NSF-AWM Mentoring Travel Grant}

\begin{document}

\begin{abstract}
We give definitions that distinguish between two notions of indiscernibility for a set $\{a_\eta \mid \eta \in \W\}$ that saw original use in \cite{sh90}, which we name \textit{$\s$-} and \textit{$\n$-indiscernibility}.  Using these definitions and detailed proofs, we prove $\s$- and $\n$-modeling theorems and give applications of these theorems.  In particular, we verify a step in the argument that TP is equivalent to TP$_1$ or TP$_2$ that has not seen explication in the literature.  In the Appendix, we exposit the proofs of \citep[{App. 2.6, 2.7}]{sh90}, expanding on the details. 
\end{abstract}

\maketitle

\section{Introduction}

Many classification-theoretic properties can be stated in terms of the existence of an infinite set of witnesses to some ``forbidden'' graph configuration, where the edge relation is some definable relation in the theory.  The following properties are all such examples: being unstable, having the independence property, having the tree property (being non-simple), having TP$_1$, TP$_2$, or the SOP$_n$, for $n = 1,2$.  Being able to choose ``very homogeneous'' witnesses to the definable configuration, whenever witnesses exist, is a very powerful tool. We look at unstable theories as an example of this situation.  A theory $T$ is \emph{unstable} just in case it has the \emph{order property}, i.e. there exist some formula $\varphi(\ov{x};\ov{y})$ in the language of $T$, and some infinite set of finite tuples $\{ \la \ov{a}_i, \ov{b}_i \ra \mid i<\omega\}$ from the monster model $\CM$ such that $\vDash \varphi(\ov{a}_i;\ov{b}_j) \Leftrightarrow i<j$.  The correct notion of a ``very homogeneous'' infinite set in this case is that of an \emph{order-indiscernible sequence}, which is a sequence of parameters $(\ov{c}_i \mid i \in I)$, indexed by some linear order $I$, such that $\vDash \psi(\ov{c}_{i_1}, \ldots, \ov{c}_{i_n}) \leftrightarrow \psi(\ov{c}_{j_1}, \ldots, \ov{c}_{j_n})$, whenever $i_1< \ldots < i_n$ and $j_1 < \ldots < j_n$ are in $I$ and $\psi(\xb_1, \cdd, \xb_n)$ is any formula in the language.  By Ramsey's theorem (and compactness), we may always choose an indiscernible sequence of witnesses to the order property in any unstable theory.

Suppose we consider the linear order $I$ as a structure in its own right, a set with a binary relation, $\I := (I,<)$.  Then a set $\{c_i : i \in \I\}$ of parameters from $\CM$ is order-indiscernible just in case for any $n$ and $(i_1, \ldots, i_n), (j_1, \ldots, j_n)$ from $I$ with the same quantifier-free type in $\I$, we have that $(c_{i_1}, \ldots, c_{i_n})$ and $(c_{j_1}, \ldots, c_{j_n})$ share the same complete type in $\CM$.  Once viewed from this perspective, one may have as many notions of indiscernibility as there are useful index structures to serve in the place of $\I$, as was first pointed out in \cite{sh90}.  If the set of indiscernible parameters is indexed by a structure $\I$ that is a tree under some expansion of the language for partial orders, $\{\unlhd\}$, call it a tree-indexed indiscernible (or specifically, an $\I$-indexed indiscernible.)  Tree-indexed indiscernibles have been  studied in several places, among them \cite{dzsh04,bash12}.  A certain tree-indexed indiscernible was recently used in \cite{kim10} to show that TP$_1$ is equivalent to $k$-TP$_1$.

In this paper we give explicit definitions for two notions of tree-indexed indiscernibility that are used in \cite{sh90}, which we name $\s$- and $\n$-indiscernibility.  These notions are used in \citep[{Thm III.7.11}]{sh90} to prove that $k$-TP (for some $k\geq 2$) is equivalent to $2$-TP$_1$ or $k'$-TP$_2$, for some $k'$. (See subsection \ref{TPs} for the definitions of $k$-TP$_i$.)    Say that a set of $\I$-indexed indiscernibles $B = \{b_i \mid i \in \I\}$ is \emph{based on} a set $A = \{a_i \mid i \in \I \}$ if for any formula $\varphi$ in the language of $\CM$ and for any $(b_{j_1}, \ldots, b_{j_n})$ there exists  some $(a_{i_1}, \ldots, a_{i_n})$ so that $(b_{j_1}, \ldots, b_{j_n}) \equiv_\varphi (a_{i_1}, \ldots, a_{i_n})$ and $\qt(i_1,\ldots,i_n;\I) = \qt(j_1,\ldots,j_n;\I)$.  (See Definition \ref{50}.)  In this sense,   $\{b_i \mid i \in \I \}$ is ``finitely modeled'' on  $\{ a_i \mid i\in \I \}$.  We say $\I$-indexed indiscernibles have the \emph{modeling property} if for any set of parameters $A = \{a_i \mid i \in \I \}$ we may find an $\I$-indexed indiscernible set $B = \{b_i \mid i \in \I\}$ based on $A$.

That $\s$- and $\n$-indiscernible sets have the modeling property is the content of Theorems \ref{s-model} and \ref{1-model}, the \emph{$\s$- and $\n$-modeling theorems}.  These claims are implicit in \citep[{Thms III.7.11, VII.3.6}]{sh90}. However explicit proofs do not appear to be present, so we supply them here.  The interested reader can read alternate proofs of the $\s$- and $\n$-modeling theorems in \cite{tats12} that rely on the fact exposited in Theorem \ref{1}.  Moreover, we define $\s$- and $\n$\emph{-type properties}, and prove in Theorem \ref{main2} that such properties are always modeled by $\s$- and $\n$-indiscernibles that are based on parameters witnessing those properties. 

It would be wrong to construe the following Claim \ref{250} from the proof of \citep[{Thm III.7.11}]{sh90}, where $\n$-indiscernible witnesses to TP are obtained from $\s$-indiscernible witnesses citing Ramsey's theorem and compactness.

\begin{clm}\label{250} Suppose $\varphi(x ; y)$ is a formula witnessing $k$-TP with some $\s$-indiscernible parameters $\{ a_\eta\mid \eta\in \leftexp{\omega >}{\omega} \}$. If  $\{ b_\eta\mid \eta\in \leftexp{\omega >}{\omega} \}$ is any set of $\n$-indiscernible parameters $\n$-based on $\{ a_\eta\}_\eta$, then $\varphi (x;y)$ and $\{ b_\eta \}_\eta$ must witness $k'$-TP for some $k'$.
\end{clm}

\nin This is because Claim \ref{250} is false, as we show in Proposition \ref{33}.  However, the conclusion of Claim \ref{250} is true under the added assumption of NTP$_2$, and so this step in the proof is valid.  We feel there is room to add the relevant details here, as it also serves to exposit the technology of tree-indexed indiscernibles.  In Lemma \ref{252}, we illustrate how the assumption of TP together with $k$-NTP$_2$ for all $k$, yields $\n$-indiscernible witnesses to TP.  The proof for Theorem \ref{tp12} summarizes the rest of the argument as it stands in the literature.

Note that the statement of  \citep[{Thm III.7.11}]{sh90} is as follows:  under the assumption of $k$-TP (for some $k\geq 2$),  the negation of $k'$-TP$_2$ for all $k'$ implies $2$-TP$_1$.  A modified statement is given in \cite{ad07}: under the assumption of $k$-TP (for some $k\geq 2$), the negation of $2$-TP$_2$ implies $2$-TP$_1$.  To obtain the modified statement, one uses the fact that 2-TP$_2$ is equivalent to $k$-TP$_2$.  This fact is argued for in \cite{ad07}. We repeat Adler's main argument in Proposition \ref{ktp2}, but we use the str-modeling theorem (Theorem \ref{1-model}) to justify the key ingredient of the proof,  namely that we may assume a given array to be `indiscernible' (see Theorem \ref{arraymodel} and Lemma \ref{main 3}). (That the str-modeling theorem can be used to simplify the argument here was pointed out to the authors by the anonymous reviewer of this paper, to whom we would like to express gratitude.)

We state notation and key propositions in section \ref{51}.  The definitions of $\s$- and $\n$-indiscernible sets are given in section \ref{34}.   
In section \ref{36} we argue for  the $\s$- and $\n$-modeling theorems, Theorems \ref{s-model} and \ref{1-model}.  In section \ref{37} we give formal definitions of $\s$- and $\n$-type properties, and give arguments for the main theorems, Theorem \ref{main2} and Theorem \ref{tp12}. In the Appendix, we supply detailed proofs of \citep[{App. 2.6, 2.7}]{sh90} for the interested reader.

We acknowledge helpful conversations with Thomas Scanlon on the subject of this paper.  We thank Artem Chernikov for suggesting Adler's paper.  We thank Miodrag Sokic for pointing out an error in the initial proofs in the Appendix.  We thank the reviewer for a careful reading of this paper and for pointing us to a simplified proof of the $\s$-modeling theorem that avoids the use of \citep[{App. 2.6, 2.7}]{sh90}, and also for suggesting simplifications of the proofs for Theorem \ref{arraymodel} and Lemma \ref{252} that eschew infinite combinatorics.

\section{Notation and conventions}\label{51}

We use standard notation. We work in a saturated model $\CM$ of a complete theory $T$ in a first-order language $L$, and $x, y, a, b, \ldots$ denote finite tuples.   We write $\vDash \varphi$ to denote that $\CM \vDash \varphi$.  Given a set $I$, a tuple $\ov{\imath}$ from $I$ is assumed to be finite unless said otherwise.  For a structure $\I$ we write $i \in \I$ when we mean $i$ is in the underlying set, $|\I|$.
For an indexed set of parameters $\{a_\eta \mid \eta\in \I\}$ and an $n$-tuple $\bar \eta$ from $\I$, let $\bar a_{\etb} := (a_{\eta_0}, \ldots, a_{\eta_{n-1}} )$, $(\etb)_l := \eta_l$, and $l(\etb):=n$.  We reserve $a\equiv b$ ($a\equiv_{\Delta} b$, resp.) to mean $\tp(a)=\tp(b)$
($\tp_{\Delta}(a)  =\tp_{\Delta}(b)$, resp.) as computed in $\CM$.  Given a structure $\I$ and a tuple $\ov{\eta}$ from $\I$, we let $\qt(\ov{\eta};\I)$ denote the quantifier-free type of $\ov{\eta}$ in $\I$ (namely, the set of all quantifier-free formulas in the language of $\I$ satisfied by $\ov{\eta}$ in $\I$.)  In the case that $I = |\I|$ for an $L'$-structure $\I$, by $\qt^{L'}(\ov{\eta};I)$ we mean $\qt(\ov{\eta};\I)$.  For a type $\Gamma$ in variables $\{x_i \mid i \in \I\}$ and an $\I$-indexed set $A =\{a_i \mid i \in \I\}$ by $A \vDash \Gamma$ we mean that $\Gamma$ is satisfied under the variable assignment $x_i \mapsto a_i$ for $i \in \I$.
For a set $X$ we write $\|X\|$ for the cardinality of $X$.    By $\leftexp{Y}{X}$ we mean the set of functions from $Y$ into $X$; for an ordinal $\beta$, $\leftexp{\beta >}{X} := \bigcup_{\al<\beta} \leftexp{\al}{X}$, $\leftexp{\beta \geq}{X} := \leftexp{\beta >}{X} \cup \leftexp{\beta}{X}$.  For $\eta \in \leftexp{\beta >}{X}$, $\ell(\eta)$ denotes the domain of $\eta$.

\subsection{Trees}

For the rest of the paper, we let $\lambda$ stand for an ordinal, and $\beta$ stand for an ordinal $\leq \omega$.  We consider $\leftexp{\omega >}{\lambda}$ as a tree under the usual partial ordering $\unlhd$: $\eta \vartriangleleft \nu$  $\Leftrightarrow$ $\eta$ is a proper initial segment of $\nu$.  Any tree $\unlhd$-isomorphic to ${}^{n\geq}\lambda$ is said to have a \emph{height $n$}. The {\em meet} of two elements $\eta, \nu \in \leftexp{\omega >}{\lambda}$ with respect to  $\unlhd$ is denoted by $\eta \wedge \nu$.   We let $\lx$ denote the lexicographic order on elements $\eta \in \leftexp{\omega >}{\lambda}$ as sequences.  i.e., $\eta \lx \nu$ means that either  $\eta  \vartriangleleft \nu$, or $\eta$ and $\nu$ are $\unlhd$-incomparable in such a way that, if $\alpha$ is the least ordinal such that $\eta(\alpha) \neq \nu(\alpha)$, then $\eta(\alpha) < \nu(\alpha)$. We shall write $\eta \len \nu$ to mean $\ell(\eta) < \ell(\nu)$.  We let $P_n(\leftexp{\omega >}{\lambda}) := \{ \eta \in \leftexp{\omega >}{\lambda} \mid \ell(\eta) = n\}$ and $P_{\leq n}(\leftexp{\omega >}{\lambda}) := \bigcup_{\al \leq n} P_\alpha(\leftexp{\omega >}{\lambda})$. The elements of $P_n(\leftexp{\omega >}{\lambda})$  are said to be \emph{on the $n$-th level of the tree $\leftexp{\omega >}{\lambda}$}.  For a cardinal $\chi$, a subtree $\I \sbb {}^{n\geq}\lam$ is said to be   ``occupying all levels $\leq n$ and $\chi$-branching" if $\la \ra \in \I$ and, for each $\eta \in \I \cap \leftexp{n >}{\lambda}$, $\|\{ \alpha < \lambda \mid \eta^{\fr}\la \alpha \ra \in \I \}\| = \chi$.

\ms

Now we define languages based on the above relations:

\begin{dfn}\label{lang} The following are languages under which $\leftexp{\omega >}{\lambda}$ may be considered as a structure, using the interpretations described in the paragraph  above.
\begin{itemize}
\item $L_{\s} = \{\unlhd, \wedge, \lx, (P_\alpha)_{\alpha<\omega}\}$ (the Shelah language)
\item $L_\n = \{\unlhd, \wedge, \lx, \len \}$ (the strong Shelah language)
\end{itemize}
\end{dfn}

\begin{ntn} For the remainder of the paper, we write $\s'$ to refer to \emph{either} of $\s$ or $\n$.
\end{ntn}

\begin{dfn}\label{struc}
\

\begin{enumerate}
\item By a \emph{subtree} of $\leftexp{\omega >}{\lambda}$ we mean a substructure in the language $\{\unlhd, \wedge\}$.
\item By the \emph{$L_{\s'}$-structure on $\leftexp{n >}{\lambda}$}, we mean the structure the set inherits as a substructure of $\leftexp{\omega >}{\lambda}$.
\end{enumerate}
\end{dfn}

\begin{rmk}\label{reduct} In Definition \ref{struc} note that any $L_{\s}$-substructure of $\leftexp{\omega>}{\lambda}$ has a natural expansion to a $\{\unlhd, \wedge, \lx, \len (P_\alpha)_{\alpha<\omega}\}$-structure, which in turn has an $L_\n$-reduct.  In this way, we can talk of taking $L_\n$-reducts of $L_{\s}$-structures.  Moreover, in any subtree of $\leftexp{\omega>}{\lambda}$, $\unlhd$ is definable using $\wedge$ and $=$, but we keep the relation for ease of reading.
\end{rmk}

Here we fix terminology for trees.  A tuple $\etb$ from a tree is said to be {\em meet-closed} if for each $i,j<l(\etb)$, there is $k< l(\etb)$ such that $\eta_k=\eta_i\wedge\eta_j$.  Elements $\eta_0, ..., \eta_{k-1}\in {}^{\omega >}\lambda\, $ are called \emph{siblings} if they are distinct elements sharing the same immediate predecessor. (i.e. there exist $\nu\in {}^{\omega >}\lambda\, $ and distinct $\, t_0, ..., t_{k-1} <\lambda\, $ such that $\, \nu^{\frown}\la t_i \ra =\eta_i\, $ for each $i<k$.) Elements $\eta_0, ..., \eta_{k-1}\in {}^{\omega >}\lambda\, $ are called \emph{distant siblings} if there exist $\, \nu\in {}^{\omega >}\lambda\, $ and distinct $\, t_0, ..., t_{k-1}<\lambda\, $ such that $\, \nu^{\frown}\la t_i \ra \lteq \eta_i\, $ for each $i<k$.   When distant siblings occur on the same level, we shall call them \emph{same-level distant siblings}.

\subsection{$k$-TP, $k$-TP$_1$, weak $k$-TP$_1$ and $k$-TP$_2$}\label{TPs}

Here we recall some definitions. A theory is said to have $k$-TP if there exist a formula $\varphi (x, y)$ and a set $\{ a_\eta \mid \eta\in {}^{\og>}\og \}$ of parameters such that $\{ \varphi (x, a_{\mu \uphp n})\mid n<\og \}$ is consistent for every $\mu\in {}^\og\og$, while for any siblings $\eta_0, \cdd, \eta_{k-1}\in {}^{\og>}\og$, $\{ \varphi (x, a_{\eta_j})\mid j <k \}$ is inconsistent. The definitions of  $k$-TP$_1$ and \emph{weak} $k$-TP$_1$ are obtained by replacing the word `siblings' by `pairwise  $\unlhd$-incomparable elements'  and `distant siblings', respectively. A theory is said to have $k$-TP$_2$ if there exist a formula $\varphi (x, y)$ and a set $\{ a^i_j \mid i, j<\og \}$ of parameters such that $\{ \varphi (x, a^{i}_{f(i)})  \mid i<\og \}$ is consistent for every function $f: \og \imp \og$, while $\{ \varphi (x, a^i_j) \mid j<\og \}$ is $k$-inconsistent for every $i<\og$. TP means $k$-TP for some $k\geq 2$. By  TP$_1$ and TP$_2$, we shall mean $2$-TP$_1$ and $2$-TP$_2$, respectively.

\section{s-indiscernibles and $\n$-indiscernibles}\label{34}

Here we make a distinction between two kinds of indiscernibility for trees presented in \cite{sh90}.
We want two notions of similarity:

\begin{dfn} For $\bar\eta, \bar\nu$ finite tuples from $\leftexp{\beta >}{\lambda}$, we make the following abbreviations
\begin{enumerate}
\item $\bar\eta \sim_{\s} \bar\nu$ (\emph{$\bar\eta$ is $\s$-similar to $\bar\nu$}) iff $\qt^{L_s}(\bar \eta; \B) = \qt^{L_s}(\bar \nu; \B)$.
\item $\bar\eta \sim_{\n} \bar\nu$ (\emph{$\bar\eta$ is $\n$-similar to $\bar\nu$}) iff  $\qt^{L_\n}(\bar \eta; \B) = \qt^{L_\n}(\bar \nu; \B)$.
\end{enumerate}
\end{dfn}

\begin{rmk}\label{251} In the case that $\ov{\eta}, \ov{\nu}$ are meet-closed tuples, $\s'$-similarity is merely $L_{\s'}$-isomorphism. Also note that $\etb \sim_{\s'} \nub$ iff $\etb' \sim_{\s'} \nub'$, where $\etb'$ and $\nub'$ denote the meet-closures of $\etb$ and $\nub$, respectively, both enumerated in the same way according to the enumerations of $\etb$ and $\nub$.
\end{rmk}

\begin{rmk} $\etb \sim_{\s} \nub$\,  implies\, $\etb \sim_{\n} \nub$.
\end{rmk}

The first part of the following definition may be found in \cite{sh90}.

\begin{dfn}\label{iindex} Fix a structure $\I$ (the \emph{index} structure.)  Given a set $\{b_i : i \in \I\}$ of same-length tuples we say it is \emph{$\I$-indexed indiscernible} if for all finite tuples $\ov{\imath}, \ov{\jmath}$ from $\I$
$$\qt(\ov{\imath};\I) = \qt(\ov{\jmath};\I)\ \Rightarrow\ \ov{b}_{\ov{\imath}} \equiv \ov{b}_{\ov{\jmath}}$$
In particular, for a set $\{b_\eta \mid \eta \in \leftexp{\beta>}{\lambda}\}$, we say it is
\begin{enumerate}
\item \emph{\s-indiscernible} (Shelah-indiscernible) if it is $\I$-indexed indiscernible for $\I$ the $L_{\s}$-structure on $\leftexp{\beta>}{\lambda}$.
\item \emph{\n-indiscernible} (strongly Shelah-indiscernible) if it is $\I$-indexed indiscernible for $\I$ the $L_{\n}$-structure on $\leftexp{\beta>}{\lambda}$.
\item (for  finite tuples $\etb_i$ from $\leftexp{\beta>}{\lambda}$, $\Delta$ a set of $L$-formulas) \emph{$\s'$-indiscernible with respect to $\ov{\eta}_0, \ldots, \ov{\eta}_m$ and $\Delta$} if for all finite tuples $\nub$ from $\leftexp{\beta >}{\lambda}$, for all $i \leq m$,

$\nub \sim_{\s'} \etb_i\, \ \Rightarrow\  \bar{b}_{\nub} \equiv_\Delta \bar{b}_{\etb_i}$
\end{enumerate}
\end{dfn}

\begin{rmk}
Every str-indiscernible set is s-indiscernible since $\sim_{\s}$ implies $\sim_{\n}$.
\end{rmk}

\begin{ntn} From now on, a set of parameters $\{a_\eta \mid \eta \in \leftexp{\beta>}{\lambda}\}$ will always consist of same-length tuples, $a_\eta$.  
\end{ntn}

With this definition we follow notation for order-indiscernibles from \citep[{Def 15.2}]{tezi11}.

\begin{dfn}\label{70} 
\begin{enumerate}
\item Let $\I$ be an arbitrary index structure. The \emph{EM-type} of a set of parameters $A = \{a_i : i \in \I\}$ is the partial type in variables $\{x_i \mid i \in \I\}$, consisting of all the formulas in the form  $\varphi(\ov{x}_{\ov{\imath}})$ (where $\ov{\imath}$ is a tuple in $\I$) satisfying the following property:
\[ \ya \varphi(\ov{a}_{\ov{\jmath}})\ \mbox{holds whenever $\ov{\jmath}$ is a tuple in $\I$ with  $\qt(\ov{\jmath}; \I)=\qt(\ov{\imath};\I)$}
\]

\smallskip

We let $\textrm{EM}_{\I}(A)$ denote this partial type.

\smallskip

\item When $\I$ is the $L_{s'}$-structure $\leftexp{\beta>}{\lambda}$, we shall write  $\textrm{EM}_{s'}(A)$ (called the \emph{$\s'$-EM-type of $A$}) for $\textrm{EM}_{\I}(A)$.
\end{enumerate}
\end{dfn}  

We recall a notion from \cite{sc12}:

\begin{dfn}\label{50} 
\begin{enumerate}
\item Let $\I$ be an arbitrary index structure. A set $B = \{b_\eta \mid \eta \in \I \}$ is \emph{based on} a set $A = \{a_\nu \mid \nu \in \I\}$ if for any formula $\varphi$ from the language of $\CM$, and for any tuple $(\eta_1, \ldots, \eta_n)$ from $\I$, there exist $(\nu_1, \ldots, \nu_n)$ from $\I$ such that
\begin{enumerate}
\item $\qt(\nu_1, \ldots, \nu_n; \I) = \qt(\eta_1, \ldots, \eta_n;\I)$, and
\item $(b_{\eta_1}, \ldots, b_{\eta_n}) \equiv_\varphi (a_{\nu_1}, \ldots, a_{\nu_n})$.
\end{enumerate}

\ms

\item  When $\I$ is the $L_{\s'}$-structure $\leftexp{\omega >}{\lambda}$, we shall say \emph{$B$ is  $\s'$-based on $A$} whenever $B$ is based on $A$.
\end{enumerate}
\end{dfn}

\begin{rmk}\label{emtp} Note that, for $B = \{b_\eta \mid \eta \in \leftexp{\omega >}{\lambda} \}$ and $A = \{a_\eta \mid \eta \in \leftexp{\omega >}{\lambda} \}$, $B$ is $\s'$-based on $A$ just in case $B \vDash \textrm{EM}_{\s'}(A)$.
\end{rmk}

\begin{rmk} We make no assumptions about the language of $\CM$.  Note that for a small subset of parameters $C \subset \CM$ we may always talk about $\I$-indexed indiscernibles \emph{over $C$}, that are based on $A$ \emph{over $C$} by simply adding constants for the elements of $C$ into the language of $\CM$.
\end{rmk}

\nin For certain tree-indexed parameters in \cite{zi88}, $B$ is defined to be \emph{lokal wie} $A$  just in case $B$ is $\s$-based on $A$.  Note that for a finite set of formulas $\Delta$, any $\Delta$-type can be conjoined into a single formula $\varphi$, so in applications $\varphi$ may be replaced by a finite set, $\Delta$.  
If a property follows from the fact that one set is based on another, we shall say  that it follows by \emph{basedness}.

\section{$\s'$-modeling theorems}\label{36}

First we verify that the usual Ramsey theorem argument works to find a sequence of infinite order-indiscernible tuples based on an initial set.

\begin{prop}\label{inf_base} Fix a possibly infinite ordinal $\gm$ and let $(\ov{a}_i)_{i<\omega}$ be a sequence where every tuple $\ov{a}_i$ has length $\gm$.  Then we may find order-indiscernible $(\ov{b}_i)_{i<\og}$ $<$-based on  $(\ov{a}_i)_{i<\og}$.
\end{prop}

\begin{proof}
Let $\Theta (\ov{x}_i\mid i<\og)$ be the EM-type of $(\ov{a}_i)_{i<\omega}$, and let $\Gamma (\ov{x}_i\mid i<\og )$ be a partial type describing order-indiscernibility of the $\ov{x}_i$. Then the usual Ramsey theorem implies that every finite subset of $\Theta \cup \Gamma$ is realizable by some infinite subsequence of $(\ov{a}_i \mid i<\og )$. Hence by compactness $\Theta \cup \Gamma$ is realized by some $(\ov{b}_i\mid i<\og )$ which is the desired sequence.
\end{proof}

\begin{rmk}
By compactness, it is easy to grow a set of $\W$-indexed $\s$-indiscernibles to an $\s$-indiscernible set indexed by $\leftexp{\omega>}{\lambda}$.  In the case of $\n$-indiscernibles, the index set can be grown to $\leftexp{\gm >}{\lambda}$ for any ordinal $\gm$.  In the next two theorems, we choose to focus on $\W$ as the index set.
\end{rmk}

For the $\s$-modeling theorem, we repeat the argument in \cite{zi88}.

\begin{thm}[$\s$-modeling theorem]\label{s-model}  Let $A = \{a_\eta \mid \eta \in \W \}$ be an $\W$-indexed set of parameters from $\CM$.  There exists $\s$-indiscernible $C = \{c_\eta \mid \eta \in \W\}$ $\s$-based on $A$.
\end{thm}

\begin{rmk} The same proof works for the case where the length of the tuple $a_\eta$ is a function of $\ell(\eta)$.
\end{rmk}

\begin{proof} By compactness, it suffices to show : given any $h<\og$ and any $A = \{ a_\eta\mid \eta\in K \}$ where $K$ is  the $L_s$-structure on ${}^{h>}\og$, there exists some $s$-indiscernible  $C = \{ c_\eta\mid \eta\in K \}$ $s$-based on $A$. We prove this by induction on $h$.

\ms

The case $h=0$ is trivial. For the induction step, suppose we are given $A= \{ a_{\eta}\mid \eta\in K\}$ where $K$ is the $L_s$-structure on ${}^{h+1>}\og$.

\ms

First, some notation. For every $m<\og$, let $K_m$ be the $L_s$-substructure of $K$ on $\{ \eta\in K \mid \la m \ra \lteq \eta \}$. Observe $K = \{\emptyset\} \cup \bigcup_{m<\omega} K_m$.  Let $A_m := \{a_\eta \mid \eta \in K_m\}$.

\ms

Using the induction hypothesis on $h$, we recursively define sets $B_0, B_1, \ldots$ where, for each $m<\og$

\vspace{.1in}

($\ast$)$_m$ ~ $B_m$ is $K_m$-indexed indiscernible over the set $\{a_\emptyset\} ~\cup~ \bigcup_{i<m} B_i \cup \; \bigcup_{k > m} A_k$, and is $\s$-based on $A_m$ over this same set.

\vspace{.1in}

Now we let $b_{\emptyset} : = a_{\emptyset}$ and  $B: = \{b_\emptyset\} \cup \; \bigcup_{i<\omega} B_i$.

\begin{clm}\label{c1}
Every $B_m$ is $s$-indiscernible over $\{ b_{\emptyset} \} \cup \bigcup_{i\neq m} B_i$.
\end{clm}

\textit{Proof of Claim \ref{c1}}. Suppose not. Then we have
\[ \vDash \varphi(b_{\ov{l}_1}, \ov{d}_{m+1}, \ldots, \ov{d}_n) \leftrightarrow \neg \varphi(b_{\ov{l}_2}, \ov{d}_{m+1}, \ldots, \ov{d}_n)
\]
where $\varphi(\ov{x}_m, \ov{x}_{m+1}, \cdd, \ov{x}_n)$ is a formula (for some $n > m$) with parameters from $\{ b_{\emptyset} \}\cup\, \bigcup_{i<m} B_i$, and $\ov{l}_1, \ov{l}_2$ are tuples from $K_m$ with the same quantifier-free type, and $\ov{d}_{m+1}, \ldots, \ov{d}_n$ are tuples where each $\ov{d}_i$ is from $B_i$.

\ms

Then by the basedness assumption in $(\ast)_n$, there exists $\ov{e}_n$ from $A_n$ such that
$$\vDash \varphi(b_{\ov{l}_1}, \ov{d}_{m+1}, \ldots, \ov{e}_n) \leftrightarrow \neg \varphi(b_{\ov{l}_2}, \ov{d}_{m+1}, \ldots, \ov{e}_n)$$

Continue applying the basedness assumption in $(\ast)_i$, for $n \geq i > m$, to obtain $\ov{e}_n, \ldots, \ov{e}_{m+1}$, where $\ov{e}_i$ are from $A_i$ and
$$\vDash \varphi(b_{\ov{l}_1}, \ov{e}_{m+1}, \ldots, \ov{e}_n) \leftrightarrow \neg \varphi(b_{\ov{l}_2}, \ov{e}_{m+1}, \ldots, \ov{e}_n)$$
But this contradicts the indiscernibility assumption in $(\ast)_m$. This completes the proof of Claim \ref{c1}.

\begin{clm}\label{c2}
$B$ is $s$-based on $A$.
\end{clm}

\textit{Proof of Claim \ref{c2}}. We will use the following.

\begin{obs}\label{helpful} If $\ov{\eta}_i, \ov{\nu}_i$ are tuples from $K_i$ for $i=1, \cdd, n$ and $\ov{\eta}_i\sim_{\s}\ov{\nu}_i$, then we have $\qt^{L_\s}(\ov{\eta}_1, \ldots,  \ov{\eta}_n / \emptyset) = \qt^{L_\s}(\ov{\nu}_1, \ldots, \ov{\nu}_n / \emptyset)$.
\end{obs}

\nin  This is due to the definition of the $K_i$'s. We need only check that binary relations in $L_\s$ over the parameter $\emptyset$ are preserved in $\nub_i, \nub_j$ when they hold of elements
across $\etb_i, \etb_j$ for $i \neq j$.  However, any $\mu_i\in K_i$ and $\mu_j\in K_j$ (with $i < j$) are $\unlhd$-incomparable, $\mu_i  \lx \mu_j$, and $\mu_i\wedge \mu_j = \emptyset$.

\ms

So it suffices to show the following :

Given any formula  $\varphi(\ov{x}_0, \cdd, \ov{x}_n)$  over $\{ b_\emptyset \}$,  and any $\ov{\eta}_i$ from $K_i$  for $i = 0, \ldots n$ such that $\vDash \varphi(b_{\ov{\eta}_0}, \ldots, b_{\ov{\eta}_n})$, there exist $\ov{\nu}_i$ from $K_i$ for $i= 0, \ldots, n$ such that $\ov{\nu}_i \sim_{\s} \ov{\eta}_i$ and $\vDash \varphi(a_{\ov{\nu}_0}, \ldots, a_{\ov{\nu}_n})$.

\ms

But we can find such $\ov{\nu}_i$'s by the same procedure as in the proof for Claim \ref{c1}, using $(\ast)_i$ repeatedly (from $i=n$ to $i=0$) to replace tuples from $B_i$ with tuples from $A_i$. This completes the proof of Claim \ref{c2}.

\bs

Now let us fix an enumeration of $\leftexp{h>}{\omega}$, say $(k_i\mid i<\gm)$ for some ordinal $\gm$. Then we can view each $B_m$ as an infinite tuple $\ov{B_m} = (d^m_i \mid i<\gm)$, where $d^m_i : = b_{\la m \ra^\frown k_i}$.  

\begin{ntn}
We will refer to the $L_\s$-isomorphism $K_m \bij K_{m'}$ given by $\la m \ra^\frown k_i \mapsto \la m' \ra^\frown k_i$ as the \emph{natural bijection}.
\end{ntn}

\begin{clm}\label{c3}
There exists a sequence  $( \ov{C_m} \mid m<\og )$ which is order-indiscernible over $\{ b_{\emptyset} \}$, and is $<$-based on $( \ov{B_m} \mid m<\og )$ over $\{ b_{\emptyset} \}$.
\end{clm}

\textit{Proof of Claim \ref{c3}}. This is a straightforward application of Proposition \ref{inf_base}.

\bs

Now consider  $C = \{ c_{\eta}\mid \eta\in K \}$ which is naturally constructed from the $C_m$'s. More precisely:  let $\ov{C_m} = (e^m_i\mid i<\gm )$ for each $m<\og$, and for each  $\eta\in K - \{\emptyset\}$, let $c_{\eta}: = e^m_i$ where $m<\og$ and $i<\gm$ are such that $\eta =\la m \ra^{\frown}k_i$. Finally we let $c_{\emptyset}: = b_{\emptyset}$.

\begin{clm}\label{c4}
$C$ is $s$-based on $B$.
\end{clm}

\textit{Proof of Claim \ref{c4}}. Suppose $\vDash \varphi (c_{\emptyset}, c_{\ov{k}_0}, \ldots, c_{\ov{k}_n})$ for some formula $\varphi(x, \ov{x}_0, \ldots, \ov{x}_n)$  and tuples $\ov{k}_m$ from $K_m$ for $m=0, \ldots, n$. Since $(\ov{C_m}\mid m<\og )$ is $<$-based on $(\ov{B_m}\mid m<\og )$ over $\{ b_\emptyset \} = \{ c_{\emptyset} \} $, there exist indices $s(0) < \ldots <s(n) <\og$ and $\ov{k}'_m\in K_{s(m)}$ for $m=0, \ldots n$, where each $\ov{k}'_m$ is the image of $\ov{k}_m$ under the natural bijection $K_m \bij K_{s(m)}$ and such that $\vDash \varphi (b_{\emptyset}, b_{\ov{k}'_0}, \ldots, b_{\ov{k}'_n})$.  By the bijection and Observation \ref{helpful},  $(\emptyset, \ov{k}_0, \ldots, \ov{k}_n) \sim_{\s} (\emptyset, \ov{k}'_0, \ldots, \ov{k}'_n)$. This completes the proof of Claim \ref{c4}.

\begin{clm}\label{c5}
Each $C_m$ is $\s$-indiscernible over $\{c_\emptyset\} \cup \bigcup_{i \neq m} C_i$.
\end{clm}

\textit{Proof of Claim \ref{c5}}. We show that $C_0$ is $\s$-indiscernible over $\{c_\emptyset\} \cup \bigcup_{i>0} C_i$; the general case is only notationally more difficult.  Suppose not.  Then for some $n<\omega$ there exist tuples $\ov{\ell}_1 \sim_{\s} \ov{\ell}_2$ from $K_0$, $\ov{k}_i$ from $K_i$ for $1 \leq i \leq n$, and a formula $\varphi (\ov{x}_0, \ldots, \ov{x}_n)$ over $\{ c_{\emptyset} \}$, such that
\[ \vDash \varphi (c_{\ov{\ell}_1}, c_{\ov{k}_1}, \ldots, c_{\ov{k}_n}) \leftrightarrow \neg \varphi (c_{\ov{\ell}_2}, c_{\ov{k}_1}, \ldots, c_{\ov{k}_n})
\]

\nin Since $(\ov{C}_i \mid i<\og )$ is $<$-based on $(\ov{B}_i \mid i<\og )$ over $\{ c_{\emptyset} \}$, there exist $s(0) < \ldots <s(n) <\og$, $\ov{\ell}'_1, \ov{\ell}'_2\in K_{s(0)}$, and $\ov{k}'_i\in K_{s(i)}$ for $i=1, \ldots n$, where $\ov{\ell}'_1, \ov{\ell}'_2$ are the images of $\ov{\ell}_1, \ov{\ell}_2$ under the natural bijection $K_0 \bij K_{s(0)}$, each $\ov{k}'_i$ is the image of $\ov{k}_i$ under the natural bijection $K_i \bij K_{s(i)}$, and such that
\[ \vDash \varphi (b_{\ov{\ell}'_1}, b_{\ov{k}'_1}, \ldots, b_{\ov{k}'_n}) \leftrightarrow  \neg \varphi (b_{\ov{\ell}'_2}, b_{\ov{k}'_1}, \ldots, b_{\ov{k}'_n})
\]

\nin But this is impossible since clearly $\ov{\ell}'_1 \sim_{\s} \ov{\ell}'_2$ and, by Claim \ref{c1}, $B_{s(0)}$ is $s$-indiscernible over $\{b_{\emptyset}\}\cup \bigcup_{i\neq s(0)} B_i$. This completes the proof of Claim \ref{c5}.

\ms

Now Claims \ref{c3} and \ref{c5} imply that $C$ is $s$-indiscernible. Moreover, $C$ is $s$-based on $B$ which is $s$-based on $A$. Hence $C$ is $s$-based on $A$, since $s$-basedness is a transitive property. This completes the induction step and the proof of Theorem \ref{s-model}.
\end{proof}

\begin{thm}[$\n$-modeling theorem]\label{1-model} Let $A = \{a_\eta \mid \eta \in \W \}$ be an $\W$-indexed set of parameters from $\CM$.  There exists $\n$-indiscernible $B = \{b_\eta \mid \eta \in \W\}$ $\n$-based on $A$.
\end{thm}

\begin{proof}
By Theorem \ref{s-model}, there
is an $\s$-indiscernible $C := \{ c_{\eta}|\ \eta\in \W \}$ $\s$-based on $A$.  By Remark \ref{reduct} and taking reducts, it is clear that $C$ is also $\n$-based on $A$, as every quantifier-free $L_\n$ type is a union of quantifier-free $L_{\s}$-types.
Since $\n$-basedness is a transitive property, it suffices to find a $\n$-indiscernible set $\{b_\eta \mid \eta \in \W\}$ that is $\n$-based on $C$.

\smallskip

By Remark \ref{emtp}, it suffices to show that the type $\textrm{EM}_\n(C) \cup \Psi( x_{\eta}|\ \eta\in \W)$ is consistent, where $\Psi( x_{\eta}|\ \eta\in \W)$ is the type describing the $\n$-indiscernibility of the $x_\eta$'s.

\smallskip

By compactness, it suffices to fix any finite set $\Delta$ of formulas and any finitely many meet-closed tuples $\etb_1, \cdd, \etb_t$ in ${}^{\og>}\og$, and assume that the type $\Psi(x_\eta\mid \eta\in {}^{\og>}\og)$ expresses the str-indiscernibility  of $\{x_\eta \mid \eta \in {}^{\og>}\og\}$ with respect to $\etb_1, \cdd, \etb_t$ and $\Delta$.

\smallskip

Given any subset $E\sbb {}^{\og>}\og$, let  $P(E):= \{ \ell(\eta) \mid \eta\in E \}$. 
(Recall $\ell(\eta)$ denotes the \textit{level} of $\eta$ in the tree ${}^{\og>}\og$.) 
 By $P$($\nub$) we mean $P$(ran($\nub$)).
In particular, $P(E) \sbb \og$.

\smallskip

For each $i=1, \cdd, t$, let $\mathcal{C}(\etb_i)$ denote the set consisting of all the tuples $\nub$ in ${}^{\og>}\og$ which are $\sim_\n$-equivalent to $\etb_i$. Note that, for each $i$,  $P(\nub)$ has a fixed size for all $\nub\in \mathcal{C}(\etb_i)$. Let $k_i$ denote this fixed size. In fact, for our purposes, it suffices to assume $k_1  = \cdd = k_t (:=k)$.

\smallskip

\begin{obs} For each $i$, the following holds: for any $\nub_1, \nub_2\in \mathcal{C}(\etb_i)$, if $P(\nub_1) = P(\nub_2)$ then $\nub_1\sim_s \nub_2$.
\end{obs}

\smallskip

Since $C=\{ c_{\eta} \mid \eta\in {}^{\og>}\og \}$ is $s$-indiscernible, we can then well-define a map which sends each subset $A\sbb \og$ of size $k$ to a tuple of uniquely determined types $($tp$_\Delta(\cb_{\nub_1}), \cdd,$ tp$_\Delta(\cb_{\nub_t}))$ where $\nub_i\in \mathcal{C}(\etb_i)$ with $P(\nub_i) = A$. Since there are only finitely many $\Delta$-types, this is a finite coloring map. Hence, by Ramsey's theorem, there exists an infinite homogeneous subset $H\sbb \og$. Then choose any str-embedding $f\colon {}^{\og>}\og \imp {}^{\og>}\og$\,  such that $P(Im(f)) \sbb H$, and define a tree $\{ d_{\eta} \mid \eta\in {}^{\og>}\og \}$ by letting $d_\eta: = c_{f(\eta)}$. Then $\{ d_{\eta} \mid \eta\in {}^{\og>}\og \}$ satisfies EM$_{\textrm{str}}(C)\cup \Psi(x_\eta\mid \eta\in {}^{\og>}\og)$, and this completes the proof.
\end{proof}

\section{Applications}\label{37}

In the following, $\W$ could easily be replaced by $\leftexp{\omega>}{\lambda}$.

\begin{dfn}
A property $P$ (for theories) is called \emph{pre-$s'$-type} if there exists a partial type $\Gamma(x_\eta \mid \eta \in \W)$ such that
\be
\item for any theory $T$, $T$ has $P$ iff $\Gamma$ is satisfiable in some model of $T$,
\item for any $A=\{a_\eta \mid \eta \in \W\}$ realizing $\Gamma$, EM$_{\s'}(A)$ $\vdash$ $\Gamma(x_\eta \mid \eta \in \W)$.
\ee

\smallskip

Given a pre-$\s'$-type property $P$, we shall say that a set of parameters $A= \{a_\eta\mid \eta\in \W\}$ \emph{witnesses} $P$ if $A$ realizes the partial type  $\Gamma (x_\eta \mid \eta\in \W)$ associated with $P$.

\smallskip

By an \emph{$\s'$-type property}, we mean a (possibly infinite) disjunction of pre-$s'$-type properties.
\end{dfn}

\begin{rmk} \label{stype}
Note that, for any formula $\varphi$, the property ``$\varphi$ witnesses $k$-TP" is a pre-$\s$-type property. And the property ``$\varphi$ witnesses (weak) $k$-TP$_1$"  is a pre-$\n$-type property. Hence we have the following:
\be
\item $k$-TP is an $\s$-type property.
\item Weak-$k$-TP$_1$ and $k$-TP$_1$ are  $\n$-type properties.
\ee
\end{rmk}

 \begin{thm}\label{main2}
Suppose a theory $T$ has an $\s'$-type property witnessed by a set of parameters $\{a_\eta \mid \eta\in \W \}$. Then we may choose  $\{a_\eta\mid \eta\in \W\}$ to be $\s'$-indiscernible.  In particular, if $T$ has $k$-TP witnessed by a formula $\varphi(x, y)$ and parameters $\{a_{\eta}|\ \eta\in \W \}$, then we may choose $\{a_{\eta}|\ \eta\in \W \}$ to be $s$-indiscernible. Similarly, if $T$ has $k$-TP$_1$ or weak $k$-TP$_1$ witnessed by a formula $\varphi(x, y)$ and parameters $\{a_{\eta}|\ \eta\in \W \}$, then we may choose $\{a_{\eta}|\ \eta\in \W \}$ to be $\n$-indiscernible.
 \end{thm}
 
 \begin{proof}
 By the definition of $\s'$-type property, $A:=\{a_\eta\mid \eta\in \W\}$ realizes a certain partial type $\Gamma (x_\eta\mid \eta\in \W )$ associated with the given pre-$\s'$-type property, and moreover $\textrm{EM}_{\s'}(A) \vdash \Gamma$. Then by the $s'$-modeling theorems (Theorems \ref{s-model}, \ref{1-model}), there exists some $\s'$-indiscernible set of parameters $B:=\{b_\eta\mid \eta\in \W\}$ that is $\s'$-based on $A$. Then $B \vDash \textrm{EM}_{\s'}(A)$ (by Remark \ref{emtp}), and hence $B \vDash \Gamma$.
 \end{proof}

For the next result, we will need the notions of \textit{indiscernible array} and \textit{array-basedness}.

\begin{Def}\label{arraydef} Let $\I$ be the structure on $\omega \times \omega$ in the language $L_{\textrm{ar}} = \{ \len, <_2 \}$ with the interpretation: $(i,j) \len (s,t) \Leftrightarrow i<s$, and $(i,j) <_2 (s,t) \Leftrightarrow (i=s)\wedge (j<t)$.

\smallskip

We refer to an $\I$-indexed indiscernible set as an \textit{indiscernible array}; we say $B$ is \textit{array-based} on $A$ if $B$ is based on $A$ as an $\I$-indexed set (see Def.s \ref{iindex}, \ref{50}.)
\end{Def}

Another application of the str-modeling theorem yields the following theorem.

\begin{thm} [array-modeling theorem] \label{arraymodel}
Given any parameters $\{ a^i_j \mid i, j\in \og \}$, there exists an indiscernible array $\{ b^i_j \mid i, j\in \og \}$ which is array-based on $\{ a^i_j \mid i, j\in \og \}$. 
\end{thm}

\noindent \textit{Proof of Theorem \ref{arraymodel}}. 
Given any $\{ a^i_j \mid i, j\in \og \}$, we shall define an embedding $f\colon  {}^{\og>}\og \imp \og \times \og$  and consider a tree $\{ c_\eta\mid \eta\in {}^{\og>}\og \}$ defined by $c_\eta:= a_{f(\eta)}$. 

\smallskip

To define $f$, we temporarily view ${}^{\og>}\og$ as an $L_{\textrm{ar}}$-structure in which $<_2$ is interpreted as   $\eta <_2 \nu \Leftrightarrow (\ell(\eta)=\ell(\nu)) \wedge (\eta <_{\textrm{lex}} \nu)$. (And $\len$ is interpreted as the level relation in the tree, as usual.) Then we define $f$ to be any $L_{\textrm{ar}}$-embedding  ${}^{\og>}\og \hookrightarrow \og\times \og$. (It is clear that such an $L_{\textrm{ar}}$-embedding exists.)

\smallskip

Now, by Theorem \ref{1-model}, there exists  some str-indiscernible tree $\{ d_\eta \mid \eta\in {}^{\og>}\og \}$ which is str-based on $\{ c_\eta \mid \eta\in {}^{\og>}\og \}$.

\smallskip

Let $\eta_i$ be the sequence of zeroes $\eta_i : 2i \ar \{0\}$ in ${}^{\og>}\og$.
Then define an array $\{ b^i_j \mid i, j\in \og \}$ by letting $b^i_j: = d_{{\eta_i}^{\smallfrown}\langle j+1 \rangle}$. One may check that  $\{ b^i_j \mid i, j\in \og \}$ is an indiscernible array which is array-based on $\{ a^i_j \mid i, j\in \og \}$. $\ \square$

\bs

As an immediate consequence, we have the following lemma:
\begin{lem}\label{main 3}
Assume a formula $\varphi(x,y)$ and parameters $\{a^i_j |\  i,j\in\omega\}$ witness $k$-TP$_2$ (in some sufficiently saturated model). Then we may assume such $\{a^i_j |\  i,j\in\omega\}$ is array-indiscernible.
\end{lem}

\begin{proof}
By Theorem \ref{arraymodel}, there exists an indiscernible array $\{ b^i_j \mid i, j\in \og \}$ which is array-based on $\{a^i_j |\  i,j\in\omega\}$. Then clearly $\varphi(x,y)$ still witnesses $k$-TP$_2$ with $\{ b^i_j \mid i, j\in \og \}$.
\end{proof}

\bs

For completeness we repeat the proof of \citep[{Prop. 13}]{ad07}.

\begin{prop}\label{ktp2}
If a formula $\varphi(x, y)$ witnesses $k$-TP$_2$ ($k\geq 2$) then there exists some finite conjunction $\psi(x, \yb) = \bigwedge_{i=1}^n \varphi(x, y_i)$ witnessing TP$_2$. Hence, a theory has TP$_2$ iff it has $k$-TP$_2$ for some $k\geq 2$.
\end{prop}

\begin{proof} We prove by induction on $k$. The case $k=2$ is trivial, so assume that the claim holds for 
$2,...,k-1$, and suppose that a formula $\varphi(x, y)$ witnesses $k$-TP$_2$ with an array $\{ a^i_j \mid i, j\in \og \}$. By Lemma \ref{main 3}, we may assume $\{ a^i_j \mid i, j<\og \}$ is array-indiscernible.

\smallskip

Case I) Assume $\{\varphi(x,a_0^i)\wedge \varphi(x,a_1^i)\mid  i\in\omega\}$ is consistent:
Then the array-indiscernibility of $\{ a^i_j \mid i, j<\og \}$ implies that the conjunction $\gamma (x, \yb) =\varphi(x,y_0)\wedge\varphi(x,y_1)$ and $b_j^i:=a_{2j}^ia_{2j+1}^i$ witness  $\lceil \frac{k}{2} \rceil$-TP$_2$. Then, by the induction hypothesis, some conjunction of $\gamma (x, \yb)$ (hence some conjunction of $\varphi(x, y)$) witnesses TP$_2$.

\smallskip

Case II) Assume $\{\varphi(x,a_0^i)\wedge \varphi(x,a_1^i) \mid i\in\omega\}$ is inconsistent:
Then, by compactness, there exists  some $n$ such that  $\{\varphi(x,a_0^i)\wedge \varphi(x,a_1^i)|\ i<n\}$
is inconsistent.  Then the array-indiscernibility of $\{ a^i_j \mid i, j<\og \}$ implies that the conjunction $\psi(x, \yb) = \varphi(x,y_0)\wedge...\wedge\varphi(x,y_{n-1})$ and $b_j^i:=a_{j}^{ni}a_j^{ni+1}...a_{j}^{ni+n-1}$ witness TP$_2$.
\end{proof}

\begin{lem}\label{252}
Suppose $\varphi (x, y)$ is a formula that does not witness $k$-TP$_2$ for any $k<\og$. If $\varphi (x, y)$ witnesses $m$-TP for some $m<\og$, then it witnesses $N$-TP for some $N<\og$ with some $\n$-indiscernible parameters $\{ b_\eta\mid \eta\in \W \}$.
\end{lem}

\begin{proof}
Assume that $\varphi (x, y)$ witnesses $m$-TP for some $m<\og$. Then by Theorem \ref{main2}, there exists some $\s$-indiscernible $A = \{a_{\eta} \mid \eta \in \W \}$ witnessing $m$-TP with $\varphi(x,y)$.

\smallskip

\underline{Observation}. Since $\varphi (x, y)$ cannot witness $m$-TP$_2$ by assumption, compactness implies that there exists a number $N < \og$ satisfying the following property:  for any array of parameters $\{ c^i_j \mid i, j<\og \}$ such that $\{ \varphi (x, c^i_j)\mid j<\og \}$ is $m$-inconsistent for every $i<\og$, there exists a function $f \colon N \imp \og$ such that $\{ \varphi (x, c^i_{f(i)}) \mid i< N \}$ is inconsistent.

\smallskip

Let us fix such $N$.  Obviously we can assume $N > m$.

\smallskip

Note that any set $\{ \eta_0, \eta_1, \ldots \} \subseteq \W$ of (at least two) same-level distant siblings induces a unique pair $(s, t)\in \og^2$ where $\ell(\eta_0\wedge \eta_1) = s$ and $\ell(\eta_0) = t$.  Then we shall call the nodes $\eta_0, \eta_1, \ldots$ \textit{$(s,t)$-distant siblings}.

\smallskip

Note that any tuples of $(s,t)$-distant siblings $(\eta_0, \ldots, \eta_k)$ and $(\nu_0, \ldots, \nu_k)$ (where the $\eta_i$'s and $\nu_i$'s are ordered in increasing $\lx$ order, respectively) are $\sim_s$-similar.

\smallskip

Now, let $\{ s < t \} \subseteq \og$ be any $2$-element subset and consider any $(s,t)$-distant siblings $\eta_0, \eta_1, \ldots$ ordered in $\lx$-increasing order.

\ms

\underline{Claim}. $\{ \varphi (x, a_{\eta_i}) \mid i<\og \}$ is $N$-inconsistent.

\smallskip

\textit{Proof of Claim}. The case $s + 1 = t$ is trivial since $\varphi (x, y)$ and $A$ are witnessing $m$-TP and $m<N$. So assume $s + 1 < t$. Let $\nu_i$ denote the immediate $\unlhd$-predecessor of $\eta_i$. i.e., $\nu_i : = \eta_i \uphp (t-1)$. Then consider the array $C = \{ c^i_j \mid i, j<\og \}$ where  $c^i_j : = a_{{\nu_i}^{\frown}\la j \ra}$. Then, for each $i<\og$, $\{ \varphi (x, c^i_j)\mid j<\og \}$ is $m$-inconsistent since $\varphi (x, y)$ is witnessing $m$-TP with $A$. Therefore, by the Observation above, there exists a function $f\colon N \imp \og$ such that $\{ \varphi (x, c^i_{f(i)}) \mid i< N \}$ is inconsistent. Then it follows that  $\{ \varphi (x, a_{\eta_i}) \mid i< \og \}$ is $N$-inconsistent by  the $s$-indiscernibility of $A$. This completes the proof of the Claim.

\ms

Notice that the choice of $(s,t)$ was arbitrary. i.e., the Claim is valid for any arbitrary set of  same-level distant siblings $\{ \eta_0, \eta_1, \ldots \}$. This means that $\varphi (x, y)$ and $A = \{ a_{\eta} \mid \eta\in \W \}$  witness a \emph{strong $N$-TP}. i.e., in addition to witnessing $N$-TP, they also satisfy: $\{ \varphi (x, a_{\eta_i})\mid i<\og \}$ is $N$-inconsistent for \textit{any} set of same-level distant siblings $\{ \eta_0, \eta_1, \ldots \}$. But  strong $N$-TP is clearly a $\n$-type property. Hence, by Theorem \ref{main2}, we may replace $A$ by some $\n$-indiscernible $\{ b_{\eta} \mid \eta\in \W \}$.
\end{proof}

\begin{thm}\label{tp12}
If a formula $\varphi(x, y)$ witnesses $k$-TP ($k\geq 2$), then there exists some
finite conjunction $\psi(x, \yb)=\bigwedge_{i=1}^n \varphi(x, y_i)$ witnessing TP$_1$ or TP$_2$. Hence, a theory has TP iff it has TP$_1$ or TP$_2$.
\end{thm}

\begin{proof}
Assume that a formula $\varphi (x, y)$ witnesses $k$-TP for some $k \geq 2$. If $\varphi(x, y)$ witnesses $m$-TP$_2$ for some $m$, then Proposition \ref{ktp2} ensures that some finite conjunction $\bigwedge_{i=1}^n \varphi(x, y_i)$ witnesses TP$_2$. On the other hand, assume that $\varphi(x, y)$ does not witness $m$-TP$_2$ for any $m$.  In this case, the argument in  \citep[{Thm 14}]{ad07} gives a clear and detailed proof that some finite conjunction $\bigwedge_{i=1}^n \varphi(x, y_i)$ witnesses TP$_1$, except that it leaves the following fact to be checked by the reader; we may assume that $\varphi(x, y)$ witnesses $k$-TP with some \emph{str-indiscernible} parameters. But this fact is guaranteed by  Lemma \ref{252} above.
\end{proof}

Finally, we show that Claim 1.1 (from the Introduction section) is false by producing a counterexample.

\begin{prop}\label{33}
Claim \ref{250} is false.
\end{prop}

\begin{proof}  To produce a counterexample, we consider the theory $T^\ast_{\textrm{feq}}$ of infinitely many independent parameterized equivalence relations $x \sim_z y$.  (See \cite{sh93}.)  A model of this theory is a two-sorted structure in sorts $Q$ and $P$, where  each $c\in Q$ gives an equivalence relation $\sim_c$ on $P$ consisting of infinitely many equivalence classes, in such a way that the following property holds: for any finitely many, distinct elements $c_1, \cdd, c_k\in Q$, and for any $b_1, \cdd, b_k\in P$, there exists some $a\in P$ such that $a\sim_{c_i} b_i$ for each $i$.

\smallskip

Now, choose any set $\{ c_{\nu} \mid \nu\in \W \}$ where the $c_\nu$'s are distinct elements of  $Q$. Also choose any $c\in Q$ distinct from all the $c_{\nu}$'s.  For each $\nu\in \W$, choose any sequence $( d_{\nu^{\smallfrown}\la i \ra})_{i<\og}$ of elements in $P$  such that $i\neq j$ implies  $d_{\nu^{\smallfrown}\la i \ra} \nsim_{c_\nu} d_{\nu^{\smallfrown}\la j \ra}$. Also let $d_{\la \ra}$ be any arbitrary element of $P$. Then define $a_{\nu^\smallfrown \la i \ra}:= (d_{\nu^\smallfrown\la i \ra} , c_\nu)$ for each $\nu\in \W$ and each $i<\og$, and also define $a_{\la \ra}: = (d_{\la \ra}, c)$.

\smallskip

Now, let $\varphi(x; y, z)$ be the formula $x\sim_{z}y$. Then it  follows that $\{ \varphi (x; a_{\eta}) \mid \eta\in \W \}$  witnesses $2$-TP. Moreover, for any finitely many, distinct elements $\eta_1, \cdd, \eta_k\in \W$ where no two $\eta_i$'s are siblings with each other, the conjunction $\bigwedge_{i=1}^k \varphi(x; a_{\eta_i})$ is consistent. Note that $2$-TP together with this property is an $s$-type property. Hence, by Theorem \ref{main2}, we may assume $\{ a_\eta\mid \eta\in \W \}$ is $s$-indiscernible.

\smallskip

To complete the proof, it remains to find some str-indiscernible $\{ b_\eta \mid \eta\in \W \}$ that is str-based on $\{ a_\eta\mid \eta\in \W \}$ such that $\{ \varphi(x, b_\eta) \mid \eta\in \W \}$ does not witness $k$-TP for any $k$. The idea is first to take a subtree $\{ a'_\eta\mid \eta\in \W \}$ of $\{ a_\eta\mid \eta\in \W \}$ as follows: define a function $h\colon \W \rightarrow \W$ by letting $h(\la \ra):= \la \ra$ and  $h(\eta^{\fr}\la i \ra) =  h(\eta)^{\fr}\la i \ra^{\fr}\la 0 \ra$ for all $\eta\in \W$ and all $i<\og$. Then define $a'_{\eta}: = a_{h(\eta)}$. Note that $\{ a'_\eta\mid \eta\in \W \}$ is str-based on $\{ a_\eta\mid \eta\in \W \}$. Moreover, note that $\{ \varphi(x, a'_\eta)\mid \eta\in \W \}$ satisfies the following `strong $\varphi$-consistency':  for any finitely many $\eta_1, \cdd, \eta_k\in \W$, the conjunction $\bigwedge_{i=1}^k \varphi(x, a'_{\eta_i})$ is consistent. Now, by the str-modeling theorem (Theorem \ref{1-model}), there exists some str-indiscernible $\{ b_\eta\mid \eta\in \W \}$ that is str-based on $\{ a'_\eta\mid \eta\in \W \}$. Then  $\{ b_\eta \mid \eta\in \W \}$ still has the strong $\varphi$-consistency. In particular, $\{ \varphi(x, b_\eta)\mid \eta\in \W \}$ does not witness $k$-TP for any $k$. Finally we note that $\{ b_\eta \mid \eta\in \W \}$ is str-based on $\{ a_\eta\mid \eta\in \W \}$ since  str-basedness is a transitive property. Hence such $\{ b_\eta \mid \eta\in \W \}$ is a counterexample to Claim \ref{250}.
\end{proof}

\section{Appendix}

Here we reproduce proofs of the combinatorial lemmas referenced in \citep[{Thm III.7.11}]{sh90}.  They are  more involved than the proof given for the $\s$-modeling theorem, but perhaps a detailed exposition will find future application.  See \cite{grsh86} for a different approach to Theorem \ref{1}.

We use the following notation: For any tuple $\etb = (\eta_l)_{l<m}$ from $\leftexp{\og>}{\lambda}$ and any $n<\og$, $\bar\eta \uphp n :=(\eta'_l)_{l<m}$ where
\[ \eta'_l: =
\begin{cases}
\eta_l & \text{if $\ell(\eta_l)\leq n$}\\
\eta_l\uphp n & \text{if $\ell(\eta_l)> n$}
\end{cases}
\]

\begin{lem}\label{qtfl}
For any tuples $\etb, \nub$ from $\leftexp{\omega>}{\lambda}$ and any $n<\og$, if\, $\etb \sim_{\textrm{s}} \nub$ then\, $\etb\uphp n \sim_{\textrm{s}} \nub\uphp n$.
\end{lem}

\begin{proof}  By Remark \ref{251}, we may assume that the tuples $\etb, \nub$ are meet-closed, and therefore the mapping $f : \eta_l\mapsto \nu_l$ is an $L_\s$-isomorphism.  Let $\etb', \nub'$ be the closure downwards under $\unlhd$ in $\leftexp{\omega>}{\lambda}$ of $\etb, \nub$, respectively, enumerated in the same way according to the enumerations of $\etb, \nub$.  It is routine to check that the natural extension of $f$ to $\etb'$ is also an $L_\s$-isomorphism mapping onto $\nub'$, because the predicates $(P_\alpha)$ preserve levels.  Thus, by restricting the domain, $f$ guarantees that $\qt^{L_\s}(\etb \uphp n; \leftexp{\omega>}{\lambda}) = \qt^{L_\s}(\nub \uphp n; \leftexp{\omega>}{\lambda})$, and so $\etb \uphp n \sim_{\s} \nub \uphp n$.
\end{proof}

In the following definitions, we shall use the  terminology introduced in Section \ref{51} (Notation and Conventions), namely a subtree $\I \sbb {}^{n\geq}\lam$ ``occupying all levels $\leq n$".

\begin{dfn}\label{RamseyPair}
\be
\item Given a pair $(n, m)$ of positive integers and an infinite cardinal $\chi$, a  cardinal $\lam$ is said to be \emph{$(n, m)$-sufficient for $\chi$} if  for every map
\[ f\colon ({}^{n\geq}\lam)^m \imp \chi
\]
there exists a $\chi^+$-branching subtree $\I \sbb {}^{n\geq}\lam$ occupying all levels $\leq n$ such that $\I$ is  $\sim_\s$-homogeneous for $f$ (i.e., any $\sim_\s$-equivalent $m$-tuples in $\I$ are sent to the same image by $f$).

\ssk

\item A pair $(n, m)$ of positive integers is called a \emph{Ramsey pair} if there exists some $k<\og$ such that, for any infinite cardinal $\chi$, $\beth_k(\chi)^+$ is $(n, m)$-sufficient for $\chi$. Let $\mathcal{R}(n, m)$ denote the smallest such $k$ for a Ramsey pair $(n, m)$.
\ee
\end{dfn}

\begin{dfn} \label{plrel}
Suppose $\bigcup_{i<\gam} X_i$ is a disjoint union of $L_\s$-structures where $\gam$ is an ordinal. In any finite Cartesian product $(\bigcup_{i<\gamma}X_i)^n$, define an equivalence relation $\pl_n$ (or simply $\pl$) as follows: $(\eta_1, \cdd, \eta_n) \pl (\nu_1, \cdd, \nu_n)$ iff for all indices $i$ and $k$,
\be
\item $\eta_i \in X_k \Leftrightarrow \nu_i \in X_k$

\smallskip

\item for any finite sequence $\eta_{i_0}, \ldots, \eta_{i_t}$  in $X_k$, $(\eta_{i_0}, \ldots, \eta_{i_t}) \sim_{\s} (\nu_{i_0}, \ldots, \nu_{i_t})$
\ee
\end{dfn}

\begin{dfn}
\be
\item Given a pair $(n, m)$ of positive integers and an infinite cardinal $\chi$, a cardinal $\lam$ is said to be \emph{strongly $(n, m)$-sufficient for $\chi$} if the following holds: given any disjoint union $\bigcup_{i<\gam} X_i$ of  $L_\s$-structures where $\gam$ is an ordinal with $\| \gam \| \leq \chi$ and each $X_i$ is $L_\s$-isomorphic to ${}^{n\geq}\lam$, for every function
\[ f\colon (\bigcup_{i<\gam}X_i)^m \imp \chi
\]
there exists a $\chi^+$-branching subtree $Y_i\sbb X_i$ (for each $i<\gam$) occupying all levels $\leq n$ such that  $\bigcup_{i<\gam}Y_i$ is $\pl$-homogeneous for $f$ (i.e., any $\pl$-equivalent $m$-tuples in $\bigcup_{i<\gam}Y_i$ are sent to the same image by $f$.)

\ssk

\item A pair $(n, m)$ of positive integers is called a \emph{strong Ramsey pair} if there exists some $k<\og$ such that, for any infinite cardinal $\chi$, $\beth_k(\chi)^+$ is strongly $(n, m)$-sufficient for $\chi$. Let $\mathcal{R'}(n, m)$ denote the smallest such $k$ for a strong Ramsey pair $(n, m)$.
\ee
\end{dfn}

\begin{rmk}
Clearly, if $\lam$ is strongly $(n, m)$-sufficient for $\chi$ then $\lam$ is $(n, m)$-sufficient for $\chi$. Hence, every strong Ramsey pair is a Ramsey pair.
\end{rmk}

The following argument is a modification of \citep[{App. 2.7}]{sh90}.

\begin{thm}\label{4}
Every Ramsey pair is a strong Ramsey pair.
\end{thm}

\begin{proof}
Let $(n, m)$ be any Ramsey pair and let $k:=\mathcal{R}(n, m)$.

\ssk

We will show  $\mathcal{R}'(n, m) \leq m(k+2)$.

\ssk

So let $\chi$ be any infinite cardinal and let $\lam: = \beth_{m(k+2)}(\chi)^+$.

\ssk

And let $\bigcup_{i<\gam}X_i$ be any disjoint union of $L_\s$-structures where $\gam$ is an ordinal with $\| \gam \| =  \chi$ and each $X_i$ is $L_\s$-isomorphic to ${}^{n\geq}\lam$.

\ssk

And fix any function
\[ f\colon (\bigcup_{i<\gam} X_i)^m \imp \chi
\]

\ms

Before we begin our proof, we need to introduce some notation.

\ssk

\be
\item For any element $\eta\in \bigcup_{i<\gam}X_i$,  $\pi(\eta)$ denotes the ordinal $i$ such that $\eta\in X_i$.

\ssk

\item Given any $m$-tuple $\etb:= (\eta_1, \cdd, \eta_m)$ in  $\bigcup_{i<\gam}X_i$ and any ordinal $\alpha$,
\[ E(\etb): = \{ \pi(\eta_1), \cdd, \pi(\eta_m)  \} \quad \mbox{and} \quad N_\alpha(\etb): = \| \{ i\in E(\etb) \mid i\geq \alpha \}\|
\]

\ssk

\item Given any $m$-tuples $(\eta_1, \cdd, \eta_m)$ and $(\nu_1, \cdd, \nu_m)$ in $\bigcup_{i<\gam}X_i$ and any ordinal $\alpha$ and any integer $t$,

\ssk

\be
\item $\etb \approx_\alpha \nub$ means:
\be
\item $\etb \pl \nub$
\item $\al \in E(\etb)$
\item if $\pi(\eta_i)\neq \al$ then $\eta_i = \nu_i$ (for all $i$).
\ee

\ms

\item $\etb \approx_{(\al, t)} \nub$ means:

\ssk

\be
\item $\etb \pl \nub$
\item $\al \in E(\etb)$
\item if $\pi(\eta_i)<\al$ then $\eta_i = \nu_i$ (for all $i$).
\item $N_\al(\etb) \leq t$
\ee
\ee

\ms

\item A subset  $\bigcup_{i<\gam} Y_i \sbb \bigcup_{i<\gam} X_i$ (where $Y_i\sbb X_i$) is called \textit{$t$-homogeneous for $f$} if, for all $i < \gam$, any $\approx_{(i, t)}$-equivalent $m$-tuples in $\bigcup_{i<\gam} Y_i$ are mapped to  the same image by $f$.
\ee

\ms

Observe that, for any subset $\bigcup_{i<\gam} Y_i \sbb \bigcup_{i<\gam} X_i$ (where $Y_i \sbb X_i$)
\be
\item $\bigcup_{i<\gam} Y_i$ is $0$-homogeneous.

\ssk

\item $\bigcup_{i<\gam} Y_i$ is $m$-homogeneous iff it is $\pl$-homogeneous.
\ee

\ms

Now, we are ready to start our proof.

\ssk

First consider the case $m=1$. Then, since $(n, m)$ is a Ramsey pair and $m(k+2)>k=\mathcal{R}(n, m)$, we readily obtain a $\chi^+$-branching subtree $Y_i\sbb X_i$ (for each $i<\gam$) occupying all levels $\leq n$ such that $Y_i$ is $\sim_\s$-homogeneous for $f$. Then it's clear that $\bigcup_{i<\gam} Y_i$ is $\pl$-homogeneous for $f$, as desired.

\ms

Next, assume $m>1$.

\ssk

Let $X^0_i: = X_i$ for each $i<\gam$. We will inductively define a descending chain
\[ \bigcup_{i<\gam}X^0_i \ \supseteq\ \bigcup_{i<\gam}X^1_i\ \supseteq\ \cdd\ \supseteq\ \bigcup_{i<\gam}X^m_i
\]

where
\be
\item $X^{j}_i$ is a $\beth_{(m-j)(k+2)}(\chi)^{+}$-branching subtree of $X^{j-1}_i$ occupying all levels $\leq n$ (for all $i<\gam$ and $0 < j \leq m$)

\ssk

\item $\bigcup_{i<\gam}X^j_i$ is $j$-homogeneous (for all $0\leq j\leq m$).
\ee

\bs

We use induction: let $0\leq q<m$ and assume that we have found a $q$-homogeneous $\bigcup_{i<\gam}X^q_i$  where each $X^q_i$ is a $\beth_{(m-q)(k+2)}(\chi)^+$-branching subtree of $X_i$ occupying all levels $\leq n$.

\ssk

We start a second induction: let $\al$ be any ordinal $<\gam$ and assume that we have defined $\{ X^{q+1}_i \mid i<\al \}$ such that
\be
\item each $X^{q+1}_i$ is a $\beth_{(m-(q+1))(k+2)}(\chi)^+$-branching subtree of $X^q_i$ occupying all levels $\leq n$,

\smallskip

\item for each $\beta<\al$, any $\approx_{(\beta, q+1)}$-equivalent $m$-tuples in
\[ \bigcup_{i\leq \beta}X^{q+1}_i \, \cup \bigcup_{\beta<i<\gam}X^q_i
\]
are mapped to the same image by $f$.
\ee

\ssk

We will find  a $\beth_{(m-(q+1))(k+2)}(\chi)^+$-branching subtree $X^{q+1}_\al\sbb X^q_\al$ occupying all levels $\leq n$ such that any $\approx_{(\al, q+1)}$-equivalent $m$-tuples in
\[ \bigcup_{i\leq \al}X^{q+1}_i \, \cup \bigcup_{\al<i<\gam}X^q_i
\]
are mapped to the same image by $f$. (This will complete the second induction, giving us a  $(q+1)$-homogeneous $\bigcup_{i<\gam} X^{q+1}_i$ which will in turn complete the first induction and indeed the entire proof of the theorem.)

\ssk

We will find such $X^{q+1}_\al$ via two claims below. But first, let us choose an arbitrary family of subtrees
\[ \{ S_i \sbb X^q_i \mid \al < i <\gam \}
\]
where each $S_i$ is a $\aleph_0$-branching subtree of $X^q_i$ occupying all levels $\leq n$.

\ms

\nin \underline{Claim 1}. There exists a $\beth_{(m-(q+1))(k+2)}(\chi)^+$-branching subtree $Z\sbb X^q_\al$ occupying all levels $\leq n$  such that any $\approx_\al$-equivalent $m$-tuples in
\[ \bigcup_{i < \al}X^{q+1}_i \, \cup\, Z\, \cup \bigcup_{\al<i<\chi}S_i
\]
are mapped to the same image by $f$.

\ms

\textit{Proof of Claim 1}. Let $\mathcal{A}:= \bigcup_{i < \al}X^{q+1}_i \, \cup\, \bigcup_{\al<i<\chi}S_i$. Note $\| \mathcal{A} \| \leq \beth_{(m-(q+1))(k+2)}(\chi)^+$.

\ssk

For any subset $B\sbb m$,  any $m$-tuple $\etb:=(\eta_0, \cdd, \eta_{m-1})$ in $X^q_\al$ and any $m$-tuple $\nub:=(\nu_0, \cdd, \nu_{m-1})$ in $\mathcal{A}$, define an $m$-tuple $\etb +_B \nub: = (\epsilon_0, \cdd, \epsilon_{m-1})$ as
\[ \epsilon_i: =
\begin{cases}
\eta_i &\text{if $i\in B$}\\
\nu_i &\text{otherwise}
\end{cases}
\]

\ms

Then, for any subset $B\sbb m$ and any $m$-tuple $\etb$ in $X^q_\al$, we can define a function
\[\Psi_B(\etb)\colon \mathcal{A}^m \imp \chi
\]
by letting $\nub \mapsto f(\etb+_B \nub)$ for each $\nub\in \mathcal{A}^m$.

\ssk

Furthermore, let $\Psi$ be the function with domain  $(X^q_\al)^m$, sending each $\etb\in (X^q_\al)^m$ to the finite string of functions $\la \Psi_B(\etb)\mid B\sbb m \ra$.

\ssk

Now, the fact $\| \mathcal{A} \| \leq \beth_{(m-(q+1))(k+2)}(\chi)^+$ implies that the range of $\Psi$ has a size at most $\tau:=\beth_{(m-(q+1))(k+2)+2}(\chi)$. Moreover, observe:
\[ \beth_{(m-q)(k+2)}(\chi) = \beth_k(\tau)
\]

So $X^q_\al$ is a $\beth_k(\tau)^+$-branching tree occupying all levels $\leq n$, and $\Psi$ can be viewed as a map $(X^q_\al)^m \imp \tau$. But  $(n, m)$ is a Ramsey pair and $k = \mathcal{R}(n, m)$. Hence, there exists a $\tau^+$-branching  subtree $Z\sbb X^q_\al$ occupying all levels $\leq n$ which is $\sim_\s$-homogeneous for $\Psi$. Then it's clear that
\[ \bigcup_{i < \al}X^{q+1}_i \, \cup\, Z\, \cup \bigcup_{\al<i<\chi}S_i
\]
is $\approx_\alpha$-homogeneous for $f$, as desired. Finally, since $\beth_{(m-(q+1))(k+2)}(\chi)^+ < \tau^+$, we may as well assume that $Z$ is $\beth_{(m-(q+1))(k+2)}(\chi)^+$-branching. This completes the proof of Claim 1.

\bs

\nin \underline{Claim 2}. Let $Z$ be as in Claim 1. Then any $\approx_{(\al, q+1)}$-equivalent $m$-tuples in
\[ \bigcup_{i < \al}X^{q+1}_i \, \cup\, Z\, \cup \bigcup_{\al<i<\gam}X^q_i
\]
are mapped to the same image by $f$.

\ms

\nin \textit{Proof of Claim 2}. Let $\etb:=(\eta_1, \cdd, \eta_m)$ and $\nub:=(\nu_1, \cdd, \nu_m)$ be any $\approx_{(\al, q+1)}$-equivalent $m$-tuples in the said set. If $N_\al(\etb) \leq q$ then the $q$-homogeneity of $\bigcup_{i<\gam}X^q_i$ ensures  $f(\etb) = f(\nub)$. So we may assume $N_\al(\etb) = q+1$.

\ssk

We define $m$-tuples $(\eta^*_1, \cdd, \eta^*_m)$ and $(\nu^*_1, \cdd, \nu^*_m)$ as follows: First, for each $i$ such that $\pi(\eta_i)\leq \alpha$, define $\eta^*_i:=\eta_i$ and $\nu^*_i:=\nu_i$.

\ssk

Next, let $\beta_1, \cdd, \beta_q$ list those ordinals in $E(\etb)$ which are strictly greater than $\al$. For each of these $\beta_i$'s,  perform the following operation. (For clarity, we will only describe the operation for $\beta_1$.)

\ssk

Let $\eta_{i_1}, \cdd, \eta_{i_t}$ and $\nu_{i_1}, \cdd, \nu_{i_t}$ list those $\eta_i$'s and $\nu_i$'s in $X^q_{\beta_1}$. In particular,
\[ (\eta_{i_1}, \cdd, \eta_{i_t}) \sim_\s (\nu_{i_1}, \cdd, \nu_{i_t})
\]
Moreover, it is clear that we can find  some tuple $(\epsilon_{i_1}, \cdd, \epsilon_{i_t})$ in $S_{\beta_1}$ such that \[(\eta_{i_1}, \cdd, \eta_{i_t}) \sim_\s (\epsilon_{i_1}, \cdd, \epsilon_{i_t}) \sim_\s (\nu_{i_1}, \cdd, \nu_{i_t})
\]

Define $\eta^*_{i_j}:=\epsilon_{i_j}$ and $\nu^*_{i_j}:=\epsilon_{i_j}$ for each $j=1, \cdd, t$.

\ms

After performing this operation for each of $\beta_1, \cdd, \beta_q$, we  obtain $m$-tuples  $\etb^*$ and $\nub^*$ in
\[  \bigcup_{i < \al}X^{q+1}_i \, \cup\, Z\, \cup \bigcup_{\al<i<\gam}S_i
\]
such that
\be
\item $\etb \approx_{(\beta_1, q)} \etb^*$ and $\nub \approx_{(\beta_1, q)} \nub^*$
\ssk

\item $\etb^* \approx_\alpha \nub^*$.
\ee

\ssk

Then the $q$-homogeneity of $\bigcup_{i<\gam}X^q_i$ ensures $f(\etb) = f(\etb^*)$ and $f(\nub) = f(\nub^*)$. Moreover, Claim 1 ensures that
\[ \bigcup_{i < \al}X^{q+1}_i \, \cup\, Z\, \cup \bigcup_{\al<i<\gam}S_i
\]
is $\approx_\al$-homogeneous for $f$, and so  $f(\etb^*) = f(\nub^*)$. So $f(\etb) = f(\nub)$. This completes the proof of Claim 2.

\ms

Hence, letting $X^{q+1}_\al:=Z$, we complete the second induction and therefore the proof of this theorem.
\end{proof}

\begin{thm} [{\citep[{App. 2.6}]{sh90}}]\label{1}
Every pair $(n, m)$ of positive integers is a Ramsey pair.
\end{thm}

\begin{proof}\hfill

\ssk

\nin \underline{Case 1} $(m=1)$. We will show $\mathcal{R}(n, m) = 0$ for every $n\geq 1$.

\ssk

Fix any integer $n\geq 1$ and any infinite cardinal $\chi$.

\ssk

Let $\lam:=  \chi^+ (=\beth_0(\chi)^+)$ and fix any function $f\colon {}^{n\geq}\lambda \rightarrow \chi$.

\ssk

Let $K_n:= {}^{n\geq}\lam$. We will inductively define a descending chain of trees
\[ K_n \supseteq K_{n-1} \supseteq \cdd \supseteq K_0
\]
where each $K_i$ is a $\chi^+$-branching tree occupying all levels $\leq n$, and satisfies the following `$i$-th height homogeneity for $f$': for any $\eta, \nu\in K_i$, if  $\ell(\eta) = \ell(\nu)$ and $\ell(\eta\wedge\nu)\geq i$ then $f(\eta) = f(\nu)$.

\ms

Observe:
\be
\item $K_n$ is $n$-th height homogeneous for $f$.
\item $0$-th height homogeneity for $f$  $\, \Leftrightarrow\  $ $\sim_\s$-homogeneity for $f$  \, (assuming $m=1$).
\ee

\ms

Hence, finding such $K_0$ would finish the proof for this case.

\ssk

We use induction: let $0< q\leq n$ and assume that we have defined a $\chi^+$-branching subtree $K_q\sbb K_n$ occupying all levels $\leq n$ which is $q$-th height homogeneous for $f$.

\ms

For each $\eta\in P_q(K_q)$,  define a function
\[ \Phi(\eta)\colon \{ i \mid q\leq i\leq n \} \imp \chi
\]
by  $i \mapsto f(\nu)$ where $\nu$ is any element in $P_i(K_q)$ such that  $\eta \unlhd \nu$. Note that the $q$-th height homogeneity of $K_q$ ensures that such a map is well-defined.

\ssk

Now, for each $\eta\in P_{q-1}(K_q)$, define:
\be
\item  $Y(\eta):= \{ \al <\lam \mid \eta^{\smallfrown}\la \alpha \ra \in K_q \}$. So $\| Y(\eta) \| = \chi^+$.
\ssk

\item $\Psi(\eta)$ is the function with domain $Y(\eta)$, sending each $\alpha \in Y(\eta)$ to $\Phi(\eta^{\smallfrown}\la \alpha \ra)$.
\ee

\ssk

Note that, for each $\eta\in P_{q-1}(K_q)$, the range of $\Psi(\eta)$ has a size at most $\chi$. Hence, by the pigeon hole principle, there exists a size-$\chi^+$ subset $Y'(\eta)\sbb Y(\eta)$ on which $\Psi(\eta)$ is constant.

\ssk

Let $B:= \{ \eta^{\smallfrown}\la \alpha \ra \mid \eta\in P_{q-1}(K_q),\ \alpha \in Y'(\eta) \}$.

\ssk

And define $K_{q-1}: = \{ \eta\in K_q \mid \eta$ is $\unlhd$-comparable with some element of $B \}$.

\ssk

Then $K_{q-1}$ is $\chi^+$-branching and occupies all levels $\leq n$, and is $(q-1)$-th height homogeneous for $f$. This completes the induction step and hence the proof for Case 1.

\bs

\nin \underline{Case 2} ($m>1$). We  prove it  by induction on $n$.

\ms

\underline{Base step} ($n=1$). Erd\H{o}s-Rado theorem  essentially states that $\mathcal{R}(1, m)\leq m-1$ for every $m\geq 1$.

\ms

\underline{Induction step}. Let $n\geq 1$ and assume that $(n, m)$ is a Ramsey pair for every $m>1$.

\ssk

Fix any $m>1$ and let $k:=\mathcal{R}(n, m)$. We  will prove $(n+1, m)$ is a Ramsey pair by showing $\mathcal{R}(n+1, m) \leq m(k+3) + 1$.

\ms

So let $\chi$ be any infinite cardinal and let $\lam:=\beth_{m(k+3)+1}(\chi)^+$. And fix any function
\[ f\colon ({}^{n+1\geq}\lambda)^m \imp \chi
\]

\ssk

We shall abbreviate  ${}^{n+1\geq}\lam$ by $K$.

\ssk

Let $\kk:=\beth_{m+1}(\chi)$ and consider the family of trees $\{ X_i \sbb K \mid i<\kk \}$ where
\[ X_i: = \{ \la i \ra^{\smallfrown}\eta \mid \eta\in {}^{n\geq}\lam \}
\]

Note that each $X_i$ is a $\lam$-branching tree of height $n$.

\ssk

For any subset $B\sbb m$ and any $m$-tuple $\etb:=(\eta_0, \cdd, \eta_{m-1})$ in $\bigcup_{i<\kk}X_i$, define an $m$-tuple $\etb_B:= (\epsilon_0, \cdd, \epsilon_{m-1})$ as
\[ \epsilon_i:=
\begin{cases}
\eta_i &\text{if $i\in B$}\\
\la \ra &\text{otherwise}
\end{cases}
\]

\ssk

Furthermore, let $f'$ be the function with domain $(\bigcup_{i<\kk}X_i)^m$, sending each $\etb\in (\bigcup_{i<\kk}X_i)^m$ to the finite string $\la f(\etb_B) \mid B\sbb m \ra$.

\ssk

Note $\chi < \kk$. So the  range of $f'$ has a size $< \kk$.

\ssk

Moreover,  observe:
\[ \beth_{m(k+3)+1}(\chi) = \beth_{m(k+2)}\beth_{(m+1)}(\chi) = \beth_{m(k+2)}(\kk)
\]

Now, recall that $(n, m)$ is a Ramsey pair and $k = \mathcal{R}(n, m)$. Furthermore, Theorem \ref{4} (and its proof) tells us that  $(n, m)$ is in fact a strong Ramsey pair and $\mathcal{R}'(n, m)\leq m(k+2)$.

\ssk

Hence, there exists a $\kk^+$-branching subtree $Y_i\sbb X_i$ (for each $i<\kk$) of height $n$ such that $\bigcup_{i<\kk} Y_i$ is $\pl$-homogeneous for $f'$. And, since  $\chi^+ < \kk^+$, we may as well assume that each $Y_i$ is $\chi^+$-branching.

\ms

Let $\I_1: = \{ \la \ra \} \cup \bigcup_{i<\kk}Y_i$.

\ms

For any $m$-tuple $\etb$ in $\bigcup_{i<\kk} Y_i$, define $\theta(\etb)\in \kk^m$ by $\theta(\etb): = ( \eta_0(0), \cdd, \eta_{m-1}(0))$.

\ms

The following two observations can be verified straightforwardly.

\ms

\nin \underline{Observation 1}.  $\I_1$ is `\emph{weakly} $\sim_\s$-homogeneous for $f$'. i.e., for any $m$-tuples  $\etb:=(\eta_0, \cdd, \eta_{m-1})$ and $\nub:=(\nu_0, \cdd, \nu_{m-1})$ in $\I_1$,
\[ \mbox{if}\quad (\etb\sim_\s \nub\ \mbox{and}\ \eta_i(0) = \nu_i(0)\ \mbox{for all $i<m$})\quad \mbox{then}\quad f(\etb) = f(\nub)
\]

\ms

\nin \underline{Observation 2}.  For any $m$-tuples $\etb, \nub$ in $\bigcup_{i<\kk}Y_i$,

\ssk

\be
\item $\etb \pl \nub\ \Leftrightarrow\ \etb\sim_\s \nub$\, and\, $\theta(\etb) = \theta(\nub)$.
\ssk

\item Suppose $\theta(\etb) = \theta(\nub)$. For any subset $B\sbb m$, if  $\etb_B = (\eta'_0, \cdd, \eta'_{m-1})$ and $\nub_B:=(\nu'_0, \cdd, \nu'_{m-1})$ then $\eta'_i(0) = \nu'_i(0)$ for all $i<m$. 

\ssk

\item whenever $\etb\sim_\s \nub$,
	\be
	\item $\theta(\etb)$ and $\theta(\nub)$ have the same order-type in $\kk$.
	
	\ssk
		
	\item $\etb_B \sim_\s \nub_B$ for every subset $B\sbb m$.
	
	\ssk
	
	\item If $\theta(\etb) = \theta(\nub)$ then $f(\etb_B) = f(\nub_B)$ for every subset $B\sbb m$.
\ee
\ee

\bs

\nin \underline{Claim}. There exists a size-$\chi^+$ subset $Z\sbb \kk$ such that $\I_2: = \{ \eta\in \I_1 \mid \eta(0)\in Z \}$  is $\sim_\s$-homogeneous for $f$.

\ms

(Note that such $\I_2\sbb K$ is $\chi^+$-branching and occupies all levels $\leq n+1$. Hence, proving this claim will finish the proof of the induction step and hence the proof for Case 2.)

\ms

\emph{Proof of Claim}. Let $\mathcal{C}$ denote the set of all $\sim_\s$-equivalent classes of $m$-tuples in $\bigcup_{i<\kk}Y_i$. Note $\| \mathcal{C} \|  = \aleph_0$.

\ssk

For any $\bar\alpha\in \kk^m$ and any subset $B\sbb m$, define a function 
\[ \mathcal{E}(\bar\alpha, B)\colon \mathcal{C} \imp \chi
\]
as follows: Let $\cc \in \mathcal{C}$. If there exists any $m$-tuple $\etb$ in  $\bigcup_{i<\kk}Y_i$ such that
\be
\item $\etb\in \cc$

\ssk

\item $\theta(\etb) = \bar\alpha$.
\ee
then define $\mathcal{E}(\bar\al, B)(\cc):= f(\etb_B)$. Otherwise, define $\mathcal{E}(\bar\al, B)(\cc):=0$.

\ms

$\mathcal{E}(\bar \alpha, B)$ is well-defined thanks to Observation 2(3)(c).

\ssk

Furthermore, let $\mathcal{E}'$ be the function with domain $\kk^m$, sending each $\bar\alpha\in \kk^m$ to the finite string $\la \mathcal{E}(\bar\alpha, B)\mid B\sbb m \ra$.

\ssk

Note that the range of $\mathcal{E}'$ has a size at most $2^\chi = \beth_1(\chi)$. Moreover,
\[ \kk = \beth_{m+1}(\chi)  \geq \beth_{m-1}(\beth_{1}(\chi))^+
\]

Hence, by Erd\H{o}s-Rado theorem, there exists a size-$(2^\chi)^+$ subset $Z\sbb \kk$ such that, for any $m$-tuples $\bar\alpha, \bar\beta$ in $\kk$, if they have the same order-type in $\kk$ then $\mathcal{E}'(\bar\alpha) = \mathcal{E}'(\bar\beta)$. And, since $\chi^+ < (2^\chi)^+$, we may as well assume $\| Z \| = \chi^+$.

\ssk

Let $\I_2:= \{ \eta\in \I_1\mid \eta(0)\in Z \}$.

\ssk

Note $\I_2$ is a $\chi^+$-branching subtree of $K$ occupying all levels $\leq n+1$.

\ssk

\nin \underline{Subclaim}. $\I_2$ is $\sim_\s$-homogeneous for $f$.

\ms

\nin \emph{Proof of Subclaim}. Let $\etb, \nub$ be any $\sim_\s$-equivalent $m$-tuples in $\I_2$. Then it's clear that we can find some $m$-tuples $\etb^*, \nub^*$ in $\bigcup_{i \in Z}Y_i$ and some subset $B\sbb m$ such that:
\be
\item $\etb^* \sim_\s \nub^*$.

\ssk

\item $\etb^*_B = \etb$\, and\, $\nub^*_B = \nub$.
\ee

\ssk

Then $\bar\alpha:=\theta(\etb^*)$ and $\bar\beta:=\theta(\nub^*)$ have the same order-type in $Z$, by Observation 2(3)(a). Hence, $\mathcal{E}'(\bar \alpha) = \mathcal{E}'(\bar{\beta})$. This implies that $\mathcal{E}(\bar\alpha, B) = \mathcal{E}(\bar \beta, B)$. Hence, if we let $\cc$ denote the $\sim_\s$-equivalence class to which both $\etb^*$ and $\nub^*$ belong, then $\mathcal{E}(\bar\alpha, B)(\cc) = \mathcal{E}(\bar \beta, B)(\cc)$. But $\mathcal{E}(\bar\alpha, B)(\cc) = f(\etb^*_B) = f(\etb)$ and $\mathcal{E}(\bar\beta, B)(\cc) = f(\nub^*_B) = f(\nub)$. Hence $f(\etb) = f(\nub)$. This completes the proof of Subclaim.

\ms

This completes the proof of Claim, and hence the induction step.
\end{proof}

\bibliographystyle{amsplain}
\bibliography{preref}

\providecommand{\bysame}{\leavevmode\hbox to3em{\hrulefill}\thinspace}
\providecommand{\MR}{\relax\ifhmode\unskip\space\fi MR }
\providecommand{\MRhref}[2]{%
  \href{http://www.ams.org/mathscinet-getitem?mr=#1}{#2}
}
\providecommand{\href}[2]{#2}
\begin{thebibliography}{10}

\bibitem{ad07}
H.~Adler, \emph{Strong theories, burden and weight},
  \url{http://www.logic.univie.ac.at/~adler/docs/strong.pdf}, 2007, {p}reprint.

\bibitem{bash12}
J.~Baldwin and S.~Shelah, \emph{The stability spectrum for classes of atomic
  models}, Journal of Mathematical Logic \textbf{12} (2012), no.~1.

\bibitem{dzsh04}
M.~D{\v{z}}amonja and S.~Shelah, \emph{On $\vartriangleleft^*$-maximality},
  Annals of Pure and Applied Logic \textbf{125} (2004), no.~1-3, 119--158.

\bibitem{grsh86}
R.~Grossberg and S.~Shelah, \emph{A nonstructure theorem for an infinitary
  theory which has the unsuperstability property},  \textbf{30} (1986), no.~2,
  364--390.

\bibitem{kim10}
B.~Kim and H.-J. Kim, \emph{Notions around tree property 1}, Annals of Pure and
  Applied Logic \textbf{162} (2011), no.~9, 698--709.

\bibitem{sc12}
Lynn Scow, \emph{Characterization of {NIP} theories by ordered
  graph-indiscernibles}, Annals of Pure and Applied Logic \textbf{163} (2012),
  no.~11, 1624--1641.

\bibitem{sh90}
S.~Shelah, \emph{Classification theory and the number of non-isomorphic models
  (revised edition)}, North-Holland, Amsterdam-New York, 1990.

\bibitem{sh93}
\bysame, \emph{The universality spectrum: consistency for more classes},
  Combinatorics, {P}aul {E}rd{\H o}s is eighty, {V}ol.\ 1, Bolyai Soc. Math.
  Stud., J\'anos Bolyai Math. Soc., Budapest, 1993, pp.~403--420.

\bibitem{tats12}
K.~Takeuchi and A.~Tsuboi, \emph{On the existence of indiscernible trees},
  Annals of Pure and Applied Logic \textbf{163} (2012), no.~12, 1891--1902.

\bibitem{tezi11}
K.~Tent and M.~Ziegler, \emph{A {C}ourse in {M}odel {T}heory}, {A}{S}{L}
  {L}ecture {N}otes in {L}ogic, {C}ambridge {U}niversity {P}ress, Cambridge,
  U.K., 2012.

\bibitem{zi88}
M.~Ziegler, \emph{Stabilit\"{a}tstheorie},
  \url{http://home.mathematik.uni-freiburg.de/ziegler/skripte/stabilit.pdf},
  1988, {n}otes.

\end{thebibliography}

\end{document}